\documentclass[12pt]{amsart}%
\usepackage{amsmath}
\usepackage{amsfonts}
\usepackage{amssymb}
\usepackage{verbatim}
\usepackage{graphicx}%
\providecommand{\U}[1]{\protect\rule{.1in}{.1in}}
\newtheorem{theorem}{Theorem}

\newtheorem{definition}[theorem]{Definition}

\newtheorem{proposition}[theorem]{Proposition}
\newtheorem{remark}[theorem]{Remark}

\numberwithin{equation}{section}
\numberwithin{theorem}{section}
\begin{document}
\title{Coarse Ricci curvature as a function on $M\times M$}
\author{Antonio G. Ache}
\address{Mathematics Department, Princeton University, Fine Hall, Washington Road,
Princeton New Jersey 08544-1000 USA }
\email{aache@math.princeton.edu}
\author{Micah W. Warren}
\address{Department of Mathematics, University of Oregon, Eugene OR 97403}
\email{micahw@oregon.edu}
\thanks{The first author was partially supported by a postdoctoral fellowship of the
National Science Foundation, award No. DMS-1204742.}
\thanks{The second author is partially supported by NSF Grant DMS-1438359. }
\maketitle

\begin{abstract}
We use the framework used by Bakry and Emery in their work on logarithmic
Sobolev inequalities to define a notion of coarse Ricci curvature on smooth
metric measure spaces alternative to the notion proposed by Y. Ollivier.
\ This function can be used to recover the Ricci tensor on smooth Riemannian
manifolds by the formula
\[
\mathrm{Ric}(\gamma^{\prime}\left(  0\right)  ,\gamma^{\prime}\left(
0\right)  )=\frac{1}{2}\frac{d^{2}}{ds^{2}}\mathrm{Ric}_{\triangle_{g}%
}(x,\gamma\left(  s\right)  )
\]
for any curve $\gamma(s).$\ 

\end{abstract}
\tableofcontents


\section{Introduction}

The Riemannian metric tensor $g$ can be recovered by differentiating the
metric function $d_{g}^{2}$ as follows. \ Given a curve $\gamma(s)$ in $M$,
with $\ \gamma^{\prime}(0)\in T_{x}M$
\[
g(\gamma^{\prime}(0),\gamma^{\prime}(0))=\frac{1}{2}\frac{d^{2}}{ds^{2}}%
d_{g}^{2}(x,\gamma(s)).
\]
In this note, we construct and discuss a similar potential for the Ricci
tensor: \ A\ function
\[
\mathrm{cRic}:M\times M\smallsetminus\mathfrak{C}\rightarrow%
\mathbb{R}
\]
(here $\mathfrak{C}$ is a \textquotedblleft cut locus" to be defined)\ such
that
\[
\mathrm{Ric}_{x}(\gamma^{\prime}(0),\gamma^{\prime}(0))=\frac{1}{2}\frac
{d^{2}}{ds^{2}}\mathrm{cRic}(x,\gamma(s)).
\]
We are most interested in functions which have following \textit{coarse Ricci
curvature lower bound }property :
\[
\mathrm{cRic}(x,y)\geq Kd_{g}^{2}(x,y)\text{ for all }x,y\in M
\]
if and only if%
\[
\mathrm{Ric}\geq Kg\text{ on }M.
\]
\ 

There is an abundance of functions that satisfy this property. \ We propose
one in particular, grounded in the Bakry-Emery $\Gamma_{2}$ calculus. In
particular, we appeal to the Bochner formula
\[
\Gamma_{2}(f,f)=\mathrm{Ric}(\nabla f,\nabla f)+\Vert\nabla^{2}f\Vert^{2}.
\]
\ The original motivation for this project was to determine if the Ricci
curvature could be recovered by using approximations of Laplace operators.
\ \ For this reason we desire a coarse Ricci function that can be constructed
from the Laplace operator and the distance function, without appealing to any
tensor calculus. \ While there are other, perhaps more naive functions which
satisfy the above properties, we use a particular one: see (\ref{RicDef}).

In \cite{BN08}, Belkin and Niyogi show that the graph Laplacian of a point
cloud of data samples taken from a submanifold in Euclidean space converges to
the Laplace-Beltrami operator on the underlying manifold. Following this
reasoning, in \cite{AW2} we define a robust family of coarse Ricci curvature
operators on metric measure spaces which depend on a scale parameter $t$.
\ Motivated by problems in manifold learning, in \cite{AW2} we show that on
smooth submanifolds of Euclidean space, these Ricci curvatures recover the
intrinsic Ricci curvature in the limit. Explicit convergence rates and
applications to manifold learning are explored in \cite{AW}.

\subsection{Background and Motivation}

\label{introBM}

The motivation for the paper stems from both the theory of Ricci lower bounds
on metric measure spaces and the theory of manifold learning.

\subsubsection{Ricci Curvature Lower Bounds on Metric Measure Spaces}

\label{introRLB}

One of our main sources of intuition for understanding Ricci curvature is the
problem of reinterpreting lower bounds on the Ricci curvature in such a way
that it becomes stable under Gromov-Hausdorff limits and thus defining ``weak"
notion of lower bounds on Ricci curvature. This theory has undergone
significant development over the last 10-15 years. A major achievement is the
work of Lott-Villani \cite{LV09} and Sturm \cite{stI,stII}, which defines
Ricci curvature lower bounds on metric measure spaces. These definitions work
quite beautifully provided the metric space is also a length space, but fail
to be useful on discrete spaces. The underlying calculus for this theory lies
in optimal transport: Given a metric measure space $(X,d,\mu)$ one can
consider the space $\mathcal{B}(X)$ of all Borel probability measures with the
$2$-Wasserstein distance, which we will denote by $W_{2}$, and is given by
\begin{align}
\label{mkdist}W_{2}(\mu_{1},\mu_{2})=\sqrt{\inf_{\gamma\in\Pi}\int_{X\times
X}d^{2}(x,y)d\gamma(x,y)},
\end{align}
where $\Pi$ is the set of all probability measures in $\mathcal{B}(X\times X)$
whose marginals are the measures $\mu_{1}$ and $\mu_{2}$, i.e. if
$P_{i}:X\times X\rightarrow X$ for $i=1,2,$ are the projections onto the first
and second factors respectively, then $\mu_{i}=(P_{i})_{*}\gamma$ (the
push-forward of $\gamma$ by $P_{i}$). One of the crucial ideas in the work of
Lott-Villani and Sturm is to use tools from optimal transport to define a
notion of convexity with respect to 2-Wasserstein geodesics (i.e., geodesics
with respect to the Wasserstein distance $W_{2}$) such that lower bounds on
the Ricci curvature are equivalent to the geodesic convexity of well-chosen
functionals. Further, they show that this convexity property is stable under
measured Gromov-Hausdorff limits. In particular, fixing a background measure
$\nu$ they define the entropy of $\mu$ with respect to $\nu$ as
\begin{align}
\label{entropy}E(\mu| \nu) = \int\frac{d\mu}{d\nu} \log\left(  \frac{d\mu
}{d\nu} \right)  d\nu,
\end{align}
and define a space to have non-negative Ricci curvature if the entropy
(\ref{entropy}) is convex along Wasserstein geodesics. This convexity property
along Wasserstein geodesics, or \emph{displacement convexity,} used by
Lott-Villani was in turn inspired by the work of Robert McCann in \cite{mc97}.
Before the work of Lott-Villani and Sturm, it was well known that lower bounds
on sectional curvature are stable under Gromov-Hausdorff limits. For more on
the theory of Alexandrov spaces, see for example \cite{bbi}.

The work of Lott-Villani and Sturm provides a precise definition of lower
bounds on Ricci Curvature on length spaces that a priori may be very rough.
More recently, Aaron Naber in \cite{Nab13} used stochastic analysis on
manifolds to characterize two-sided bounds on non-smooth spaces. Naber's
approach also relies heavily on the existence of Lipschitz geodesics. On the
other hand, it is easy to show (for example, consider the space with two
positively measured points unit distance apart) that the Wasserstein space
$W_{2}$ for discrete spaces does not admit any Lipschitz geodesics. To
overcome the lack of Wasserstein geodesics on discrete space, Bonciocat and
Sturm \cite{BS09} introduced the notion of an``$h$-rough geodesic," which
behaves like a geodesic up to an error $h.$ Then using optimal transport
methods, they are able to define a lower bound on Ricci curvature, which
depends on the scale $h.$

The notion of \emph{coarse Ricci curvature} was developed to characterize
Ricci curvature, including lower bounds, on a more general class of spaces.
The first definition of coarse Ricci curvature was proposed by Ollivier
\cite{Oll09} as a function on pairs of points in a metric measure space.
Heuristically, two points should have positive coarse Ricci curvature if
geodesics balls near the points are \textquotedblleft closer" to each other
than the points themselves. One possible precise definition of
\textquotedblleft close" is given by (\ref{mkdist}). Comparing the distance
between two points to the distance of normalized unit balls allows one to
extract useful information about the geometry of the space. It is also natural
to consider the heat kernel for some small positive time around each given
point. In fact, the intuition becomes blatant in the face of \cite[Cor. 1.4
(x)]{vRS}, which states that the distance (\ref{mkdist}) decays at an
exponential rate given by the Ricci lower bound, as mass spreads out from two
points via Brownian motion.

Even earlier, it was shown by Jordan, Kinderlehrer and Otto \cite{JKO} that
the gradient flow of the entropy function (\ref{entropy}) on $W_{2}$-space
agrees with the heat flow on $L^{2}.$ On nice metric measure spaces, one
expects this flow to converge to an invariant measure, the same one from which
the entropy was originally defined. Thus in principle, the behavior of the
entropy functional (\ref{entropy}) and the generator of heat flow are
fundamentally related. This deep relationship is explored in generality in the
recent paper \cite{AGS12}. It is natural then to attempt to define Ricci
curvature in terms of a Markov process. In fact, Ollivier's idea was to
compare the distance between two points to the distance between the point
masses after one step in the Markov process. This is also essentially the idea
in Lin-Lu-Yau \cite{LLY11}, where a lower bound on Ricci curvature of graphs
is defined. In contrast, a $\Gamma_{2}$ approach was used by Lin-Yau to define
Ricci curvature lower bounds on graphs in \cite{LY10}.

Recently, Erbas and Maas \cite{EM12} and Mielke \cite{Mi13} have provided a
very natural way to define Ricci curvature for arbitrary Markov chains. This
involves creating a new Wasserstein space, called the discrete transportation
metric, in which one can implement the idea of Lott-Sturm-Villani, i.e.,
relating lower bounds on Ricci curvature to geodesic convexity of certain
entropy functional.
Gigli and Maas \cite{GM13} show that in the limit these transportation metrics
converge to the Wasserstein space of the manifold in question as the mesh size
goes to zero, at least on the torus. Thus this notion of Erbas and Maas
recovers the Ricci curvature in the limit.

Any discussion of lower bounds on Ricci curvature would not be complete
without a discussion of isoperimetric inequalities, most notably the
log-Sobolev inequality. These inequalities, which hold for positive Ricci
lower bounds, were generalized by Lott-Villani. In the coarse Ricci setting,
log-Sobolev inequalities have been proved by Ollivier, Lin-Yau and Erbar-Maas
for their respective definitions of Ricci curvature lower bounds. We will
mention the log-Sobolev inequality that holds in general when a Bakry-Emery
condition is present. \ This Bakry-Emery type condition in principle, should
be related to coarse Ricci curvature, however, it is not clear to us in the
non-Riemannian case what the relationship should be. \ 

For further introduction to concepts of coarse Ricci curvature see the survey
of Ollivier \cite{SurveyO}.

\subsubsection{The Manifold Learning Problem}

\label{introML} Roughly speaking, the manifold learning problem deals with
inferring or predicting geometric information from a manifold if one is only
given a point cloud on the manifold, i.e., a sample of points drawn from the
manifold at random according to a certain distribution, without any further
information. From a pure mathematical perspective, a point cloud can be a
metric measure space that approximates a manifold in a measured
Gromov-Hausdorff sense despite being a discrete set. Belkin and Niyogi showed
in \cite{BN08} that given a uniformly distributed point cloud on $\Sigma$
there is a 1-parameter family of operators $L_{t},$ which converge to the
Laplace-Beltrami operator $\Delta_{g}$ on the submanifold. More generally, the
results in \cite{CoLaf06} and \cite{SW12} show that it is possible to recover
a whole family of operators that include the Fokker-Planck operator and the
weighted Laplacian $\Delta_{\rho}f=\Delta f-\langle\nabla\rho,\nabla f\rangle$
associated to the smooth metric measure space $(M,g,e^{-\rho}d\mathrm{vol})$,
where $\rho$ is a smooth function.

In \cite{AW2} we consider the problem of learning the Ricci curvature of an
embedded submanifold $\Sigma$ of $\mathbb{R}^{N}$ at a point from purely
measure and distance considerations. In \cite{AW} we will show that one
expects these notions to converge almost surely with an appropriate choice of scale.

\section{Carr\'{e} Du Champ and Bochner Formula}

Let $P_{t}$ be a $1$-parameter family of operators of the form
\[
P_{t}f(x)=\int_{M}f(y)p_{t}(x,dy),
\]
where $f$ is a bounded measurable function defined on $M$ and $p_{t}(x,dy)$ is
a non-negative kernel. We assume that $P_{t}$ satisfies the semi-group
property, i.e.
\begin{align}
P_{t+s} &  =P_{t}\circ P_{s}.\\
P_{0} &  =\mathrm{Id}.
\end{align}
In $\mathbb{R}^{n}$, an example of $P_{t}$ is the \emph{Brownian motion},
defined by the density
\[
p_{t}(x,dy)=\frac{1}{(2\pi t)^{n/2}}e^{-\frac{|x-y|^{2}}{2t}}dy,t\geq0.
\]
If now $P_{t}$ is a diffusion semi-group defined on $(M,g)$, we let $L$ be the
infinitesimal generator of $P_{t},$ which is densely defined in $L^{2}$ by%

\begin{align}
Lf=\lim_{t\rightarrow0}t^{-1}(P_{t}f-f).
\end{align}

We consider a bilinear form which has been introduced in potential theory by
J.P. Roth \cite{Roth74} and by Kunita in probability theory \cite{Kunita69}
and measures the failure of $L$ from satisfying the Leibnitz rule. This
bilinear form, the \emph{Carr\'{e} du Champ} , is defined as
\[
\Gamma(L,u,v)=\frac{1}{2}\left(  L(uv)-L(u)v-uL(v)\right)  .
\]

When $L$ is the rough Laplacian with respect to the metric $g$, then
\[
\Gamma(\Delta_{g},u,v)=\langle\nabla u,\nabla v\rangle.
\]
We will also consider the \emph{iterated Carr\'{e} du Champ} introduced by
Bakry and Emery \cite{BE85} denoted by $\Gamma_{2}$ and defined by
\[
\Gamma_{2}(L,u,v)=\frac{1}{2}\left(  L(\Gamma(L,u,v))-\Gamma(L,Lu,v)-\Gamma
(L,u,Lv)\right)  .
\]

Note that if we restrict our attention to the case $L=\Delta_{g}$ the Bochner
formula yields
\begin{align}
\Gamma_{2}(\Delta_{g},u,v)  &  =\frac{1}{2}\Delta\langle\nabla u,\nabla
v\rangle_{g}-\frac{1}{2}\langle\nabla\Delta_{g}u,\nabla v\rangle_{g}-\frac
{1}{2}\langle\nabla u,\nabla\Delta_{g}v\rangle_{g}\\
&  =\mathrm{Ric}(\nabla u,\nabla v)+\langle\mathrm{Hess}_{u},\mathrm{Hess}%
_{v}\rangle_{g}.
\end{align}

The fundamental observation of Bakry and Emery is that the properties of Ricci
curvature lower bounds can be observed and exploited by using the bilinear
form $\Gamma_{2}.$ With this in mind, they define a curvature-dimension
condition for an operator $L$ on a space $X$ as follows. If there exist
measurable functions
$k:X\rightarrow\mathbb{R}$ and $N:X\rightarrow[1,\infty]$ such that
for every $f$ on a set of functions dense in $L^{2}(X,d\nu)$ the inequality
\begin{align}
\label{cdkn}\Gamma_{2}(L,f,f)\ge\frac{1}{N}(Lf)^{2}+k\Gamma(L,f,f)
\end{align}
holds, then the space $X$ together with the operator $L$ satisfies the
$CD(k,N)$ \emph{condition}, where $k$ stands for curvature and $N$ for
dimension. In particular, when considering a smooth metric measure space
$(M^{n},g,e^{-\rho}d\mathrm{vol})$ one has the natural diffusion operator
\begin{align}
\Delta_{\rho} u=\Delta u-\langle\nabla\rho,\nabla u\rangle,
\end{align}
corresponding to the variation of the Dirichlet energy with respect to the
measure $e^{-\rho}d\mathrm{vol}$. By studying the properties of $\Delta_{\rho
}$, Bakry and Emery arrive at the following dimension and weight dependent
definition of the Ricci tensor:
\begin{align}
\mathrm{Ric}_{N}  &  =\left\{
\begin{array}
[c]{ll}%
\mathrm{Ric}+\mathrm{Hess}_{\rho} & \text{if}~N=\infty,\\
\mathrm{Ric}+\mathrm{Hess}_{\rho}-\frac{1}{N-n}(d\rho\otimes d\rho) &
\text{if}~n<N<\infty,\\
\mathrm{Ric}+\mathrm{Hess}_{\rho}-\infty(d\rho\otimes d\rho) & \text{if}%
~N=n,\\
-\infty & \text{if}~N<n,
\end{array}
\right.
\end{align}
and moreover, they showed the equivalence between the $CD(k,N)$ condition
\eqref{cdkn} and the bound $\mathrm{Ric}_{N}\ge k$.

\subsection{Iterated Carr\'{e} du Champ and Coarse Ricci Curvature}

In this section we provide a definition of coarse Ricci curvature on general
metric spaces with a given operator. In most cases of interest, this operator
will be invariant with respect to a measure on the space. The definition will
provide a coarse Ricci function given any operator.

Consider the function
\begin{equation}
f_{x,y}(z)=\frac{1}{2}\left(  d^{2}(x,y)-d^{2}(y,z)+d^{2}(z,x)\right)
\label{testdef}%
\end{equation}
One can check that, in the Euclidean case, this is simply the linear function
with gradient $x-y.$ One can also check that for $x$ very near $y$ on a
Riemannian manifold, this is (up to high order) the corresponding function in
normal coordinates at $x$ determined by the normal coordinates of $y.$ This
leads us to the following definition of coarse Ricci curvature. \ 

First, we need the following definition of cut locus, which will work for an
arbitrary metric space with an operator $L.$

\begin{definition}
For a distance function on an arbitrary metric space with an operator $L$, we
define the cut locus as
\[
\mathfrak{C=}\left\{  \left(  x,y\right)  \in X\times X:\Gamma_{2}%
(L,f_{x,y},f_{x,y})\text{ is not defined at either of }x\text{ or }y\right\}
\]

\end{definition}

\begin{remark}
This agrees with the definition of cut locus on Riemannian manifolds. \ 
\end{remark}

\begin{definition}
Given an operator $L$ we define the coarse Ricci curvature for $L$ as follows.
\ If $\left(  x,y\right)  \notin\mathfrak{C}$ we define
\begin{equation}
\mathrm{Ric}_{L}(x,y)=\Gamma_{2}(L,f_{x,y},f_{x,y})(x). \label{RicDef}%
\end{equation}

\end{definition}

In order to check that this is consistent with the classical notions, note
that this defines a coarse Ricci curvature on a Riemannian manifold as%
\[
\mathrm{Ric}_{\triangle_{g}}(x,y)=\Gamma_{2}(\Delta_{g},f_{x,y},f_{x,y})(x).
\]
We record that, in particular, $\mathrm{Ric}_{\Delta_{g}}$ has the following
property of coarse Ricci curvature on smooth Riemannian manifolds.

\begin{theorem}
\label{coarse-to-ricci} Suppose that $M$ is a smooth Riemannian manifold. Let
$\gamma\left(  s\right)  $ be a smooth curve with $\gamma^{\prime}\left(
0\right)  \in T_{x}M.$ \ Then
\begin{equation}
\mathrm{Ric}(\gamma^{\prime}\left(  0\right)  ,\gamma^{\prime}\left(
0\right)  )=\frac{1}{2}\frac{d^{2}}{ds^{2}}\mathrm{Ric}_{\triangle_{g}%
}(x,\gamma\left(  s\right)  ). \label{RecoverRicci}%
\end{equation}

\end{theorem}

\begin{remark}
\emph{{ The function $\mathrm{Ric}_{L}(x,y)$ produced by (\ref{RicDef}) need
not be symmetric. \ Also note that \emph{ (\ref{testdef})} does not require
that the distance function be symmetric. \ } }
\end{remark}

A classical result of Synge in \cite{syn31} provides an expansion for the
square of the geodesic distance at a point in normal coordinates. \ This
allows us to prove the following proposition. \ The proof of Proposition
\ref{coarse-to-ricci} will follow. \ 

\begin{proposition}
Given points $x,y\in M$ let $Y$ represent the normal coordinate of $y$ in the
tangent space at $x,$ i.e. $\ y=\exp_{x}(Y).$ \ \ Then, \qquad%
\[
\mathrm{Ric}_{\Delta_{g}}(x,y)=\mathrm{Ric}(Y,Y)+\tilde{G}(Y,x)
\]
where $\tilde{G}(Y,x)$ vanishes to fourth order at $Y=0.$
\end{proposition}

\begin{proof}
In general, for a function $h,~$we can compute in normal coordinates at
$(x=0)$
\[
\Gamma_{2}(h,h)(0)=\sum_{i,j}\left(  h_{ij}^{2}+\mathrm{Ric}_{ij}h_{i}%
h_{j}\right)
\]
In normal coordinates at $x,$ for $f_{x,y}$ given by (\ref{testdef}) the
coordinate derivatives can be computed as
\begin{align*}
\left(  f_{x,y}\right)  _{i} &  =-\left(  \frac{1}{2}d^{2}(y,z)\right)
_{i}+z_{i}\\
\left(  f_{x,y}\right)  _{ij} &  =-\left(  \frac{1}{2}d^{2}(y,z)\right)
_{ij}+\delta_{ij}%
\end{align*}
so we have at the origin $z=0$
\[
-\left(  \frac{1}{2}d^{2}(y,z)\right)  _{i}=Y^{i}%
\]%
\[
\Gamma_{2}(f_{x,y},f_{x,y})(0)=\sum_{i,j}\left(  \left[  \delta_{ij}-\left(
\frac{1}{2}d^{2}(y,z)\right)  _{ij}\right]  ^{2}+Ric_{ij}Y^{i}Y^{j}\right)
\]
in other words,
\[
\mathrm{Ric}_{\Delta_{g}}(x,y)=\mathrm{Ric}(Y,Y)+\sum_{i,j}\left[  \delta
_{ij}-\left(  \frac{1}{2}d^{2}(y,z)\right)  _{ij}\right]  ^{2}%
\]
where $Y$ is the coordinate of $y$ in normal coordinates at $x.$ \ 

Following for example, \cite[158--161]{PP} or originally \cite{syn31} , we see
that%
\[
\left(  \frac{1}{2}d^{2}(y,z)\right)  _{ij}=\delta_{ij}+O(Y^{2})
\]
and the conclusion follows.
\end{proof}

\bigskip We now prove Theorem \ref{coarse-to-ricci} .

\begin{proof}
Compute using normal coordinates with $Y=\gamma(s):$%
\[
\mathrm{Ric}_{\triangle_{g}}(x,\gamma\left(  s\right)  )=\mathrm{Ric}%
(\gamma\left(  s\right)  ,\gamma\left(  s\right)  )+\sum_{i,j}\left[
\delta_{ij}-\left(  \frac{1}{2}d^{2}(\gamma\left(  s\right)  ,z)\right)
_{ij}\right]  ^{2}.
\]
Differentiate this%
\[
\frac{d}{ds}\mathrm{Ric}_{\triangle_{g}}(x,\gamma\left(  s\right)
)=2\mathrm{Ric}(\gamma^{\prime}\left(  s\right)  ,\gamma\left(  s\right)
)+\frac{d}{ds}\left(  \tilde{G}(\gamma\left(  s\right)  ,x)\right)  ,
\]
and again%
\begin{align}
\frac{d^{2}}{ds^{2}}\mathrm{Ric}_{\triangle_{g}}(x,\gamma\left(  s\right)  )
&  =2\mathrm{Ric}(\gamma^{\prime}\left(  s\right)  ,\gamma^{\prime}\left(
s\right)  )+2\mathrm{Ric}(\gamma\left(  s\right)  ,\gamma^{\prime\prime
}\left(  s\right)  )\label{secondder}\\
&  +\frac{d^{2}}{ds^{2}}\left(  \tilde{G}(\gamma\left(  s\right)  ,x)\right)
.
\end{align}
Then plugging in $\gamma\left(  0\right)  =0$ we get%
\[
\frac{d^{2}}{ds^{2}}\mathrm{Ric}_{\triangle_{g}}(x,\gamma\left(  s\right)
)=2\mathrm{Ric}(\gamma^{\prime}\left(  s\right)  ,\gamma^{\prime}\left(
s\right)  ).
\]

\end{proof}

The following is natural, considering the analogy between coarse Ricci
curvature and the distance function. \ 

\begin{theorem}
Suppose that $\left(  M,g\right)  $ is a Riemannian manifold. \ Then
\begin{equation}
\mathrm{Ric}\geq K\label{RicLB}%
\end{equation}
if and only if%
\[
\mathrm{Ric}_{\Delta_{g}}(x,y)\geq Kd^{2}(x,y).
\]

\end{theorem}

\begin{proof}
First we show that the first condition implies the second. \ It follows from
the Bochner formula that (\ref{RicLB}) implies the Bakry-Emery condition%
\[
\Gamma_{2}(f,f)\geq K\Gamma(f,f)
\]
in particular
\begin{align*}
\mathrm{Ric}_{\Delta_{g}}(x,y)  &  =\Gamma_{2}(\Delta_{g},f_{x,y}%
,f_{x,y})(x)\\
&  \geq K\Gamma(\Delta_{g},f_{x,y},f_{x,y})(x)\\
&  =\left\vert \nabla f_{x,y}\right\vert ^{2}(x).
\end{align*}
In normal coordinates at $x$ it is easy to compute the gradient%
\[
\nabla f_{x,y}(x)=-Y
\]
so
\[
\left\Vert \nabla f_{x,y}(x)\right\Vert ^{2}=\left\Vert Y\right\Vert
^{2}=d^{2}(x,y)
\]
and the conclusion follows.

The reverse implication follows from (\ref{RecoverRicci}).
\end{proof}

\begin{theorem}
\bigskip Let
\[
\Delta_{\rho}v=\Delta_{g}v-\langle\nabla\rho,\nabla v\rangle_{g}%
\]
be the weighted Laplacian and let
\[
\mathrm{Ric}_{\infty}=\mathrm{Ric+}\nabla_{g}^{2}\rho.
\]

\end{theorem}

\begin{proposition}
Then
\[
\mathrm{Ric}_{\infty}(\gamma^{\prime}\left(  0\right)  ,\gamma^{\prime}\left(
0\right)  )=\frac{1}{2}\frac{d^{2}}{ds^{2}}\mathrm{Ric}_{\Delta_{\rho}%
}(x,\gamma\left(  s\right)  ).
\]
and%
\begin{equation}
\mathrm{Ric}_{\infty}\geq K
\end{equation}
if and only if%
\[
\mathrm{Ric}_{\Delta_{\rho}}(x,y)\geq Kd^{2}(x,y).
\]

\end{proposition}

\begin{proof}
This follows by gently modifying the above proofs, considering that the
Bochner formula becomes
\[
\Gamma_{2}(\Delta_{\rho},f,f)=\left\Vert \nabla_{g}^{2}f\right\Vert
^{2}+\mathrm{Ric}_{\infty}(\nabla f,\nabla f).
\]

\end{proof}

\section{Final Remarks}

\label{introFR}

At this point we can make comparisons to other definitions of coarse Ricci
curvature. Lin and Yau \cite{LY10}, following Chung and Yau \cite{CY96},
consider curvature dimension lower bounds of the form (\ref{cdkn}), in the
particular case where $L$ is the graph Laplacian using the standard distance
function on graphs. This class of metric spaces is quite restricted - in their
setting they are able to show that every locally finite graph satisfies a
$CD(2,-1)$ condition. The approach by Lin, Lu and Yau in \cite{LLY11} follows
the ideas of Ollivier \cite{Oll09}, using the $1$-Wasserstein distance
(denoted by $W_{1}$) instead of the $2$-Wasserstein distance. With this
metric, they compare mass distributions after short diffusion times, and take
the limit as the diffusion time approaches zero. They are able to show a
Bonnet-Myers type theorem which holds for graphs with positive Ricci
curvature. While the Bonnet-Myers result holds in the classical Riemannian
setting, it fails for general notions of Ricci curvature, for example, when
the Ricci curvature is derived from the standard Ornstein-Uhlenbeck process.
Ollivier's definition \cite[Definition 3]{Oll09} is very general, and also
uses the $W_{1}$-Wasserstein metric. Ollivier uses $\varepsilon$-geodesics to
obtain local-to-global results. Note that the notion of $\varepsilon
$-geodesics is stronger than $h$-rough geodesics seen in \cite{BS09}.

\subsection{Other definitions on Riemannian manifolds.}

One may also consider the following
\begin{equation}
\mathrm{cRic}(x,y)=\int_{0}^{1}\mathrm{Ric}_{\gamma(t)}(\dot{\gamma}%
(t),\dot{\gamma}(t))dt\label{icRic}%
\end{equation}
where $\gamma:[0,1]$ is the unique constant speed geodesic starting at $x$ and
ending at $y.$ \ \ One can check that this satisfies the coarse Ricci
curvature lower bound property. \ \ Further it satisfies the much nicer
property that if $g$ evolves by Ricci flow, i.e.
\[
\frac{d}{dt}g_{ij}=-2\mathrm{Ric}_{ij}%
\]
then
\begin{align*}
\frac{d}{dt}d^{2}(x,y) &  =\frac{d}{dt}\int_{0}^{1}g_{\gamma(t)}(\dot{\gamma
}(t),\dot{\gamma}(t))dt\\
&  =-\int_{0}^{1}2\mathrm{Ric}_{\gamma(t)}(\dot{\gamma}(t),\dot{\gamma
}(t))dt=-2\mathrm{cRic}(x,y)
\end{align*}
so this definition behaves nicely with respect to the Ricci flow. \ This
definition assumes existence and full knowledge of the Ricci tensor, so it may
be of little use for many purposes. \ 

\subsection{Lower Bounds and Log Sobolev inequalities}

From the theory of Gross in \cite{gross} and Bakry-Emery, one can show that
whenever a heat semi-group together with the invariant measure satisfies
standard properties (self adjointness, ergodicity, Leibnitz rule)\ a
Bakry-Emery condition of the form
\begin{equation}
\Gamma_{2}(f,f)\geq K\Gamma(f,f) \label{BELSCondition}%
\end{equation}
implies a log-Sobolev inequality of the form
\[
\int f\log fd\mu\leq\frac{1}{2K}\int\frac{\Gamma(f,f)}{f}d\mu
\]
for all $f>0$ with $\int fd\mu=1.$

Note that if a condition
\begin{equation}
\Gamma(L,f_{x,y},f_{x,y})(x)\geq d^{2}(x,y) \label{distcarre}%
\end{equation}
holds, then clearly (\ref{BELSCondition}) implies
\begin{equation}
\mathrm{Ric}_{L_{t}}(x,y)\geq Kd^{2}(x,y). \label{RLB}%
\end{equation}
However, it is not clear to us in general when the condition (\ref{distcarre})
holds or whether the condition (\ref{RLB}) implies a condition of the form
(\ref{BELSCondition}).

\bibliographystyle{amsalpha}
\bibliography{coarse_ricci}

\end{document}